\documentclass[a4paper,12pt]{amsart}
\usepackage[cm]{fullpage}
\usepackage{hyperref}

\usepackage{mymacros,oitp}
\usepackage{todonotes}
\includecomment{comment}

\title{AS-regular algebras from acyclic spherical helices}
\author{Shinnosuke Okawa}
\address{
Department of Mathematics,
Graduate School of Science,
Osaka University,
Machikaneyama 1-1,
Toyonaka,
Osaka,
560-0043,
Japan.
}
\email{okawa@math.sci.osaka-u.ac.jp}

\author{Kazushi Ueda}
\address{
Graduate School of Mathematical Sciences,
The University of Tokyo,
3-8-1 Komaba,
Meguro-ku,
Tokyo,
153-8914,
Japan.}
\email{kazushi@ms.u-tokyo.ac.jp}
\date{}
\begin{document}

%
%

\begin{abstract}
We discuss a
construction
of AS-regular algebras
from acyclic spherical helices,
generalizing works
of Bondal and Polishchuk
\cite{MR1230966}
and Van den Bergh
\cite{MR2836401}.
\end{abstract}

\maketitle


Let $\cC$ be an idempotent-complete pretriangulated dg category
over a field $\bk$.
See e.g.~\cite{MR2275593} and references therein
for basic definitions on dg categories.
For $X,Y \in \Ob(\cC)$,
we write the dg vector space of morphisms
in $\cC$
and its $i$-th cohomology
as $\hom(X,Y)$
and $\Hom^i(X,Y)$ respectively.

A \emph{$\cC$-module} is
a dg functor $\cC^\op \to \Module \bk$
to the dg category of dg vector spaces.
A dg functor $F \colon \cC \to \cC'$
will be identified with its \emph{graph bimodule}
$
\Gamma_F \in \Module (\cC^\op \otimes \cC').
$

Assume that $\cC$ is \emph{proper}
in the sense that
$\hom(X,Y) \in \per \bk$
for any $X, Y \in \Ob(\cC)$.
Assume further that $\cC$ has a Serre functor $\bS$
in the sense of \cite{MR1039961}.
The graph of the Serre functor is given by
\begin{align}
  \hom(Y, \bS(X)) = \hom(X,Y)^\kdual,
\end{align}
where
$
(-)^\kdual \coloneqq \hom_{\bk}(-, \bk)
 \colon (\Module \bk)^\op \to \Module \bk.
$

A prototypical example
of an idempotent-complete proper pretriangulated dg category
admitting a Serre functor
is the category of perfect complexes
on a proper Gorenstein scheme $X$.
The Serre functor is the tensor product
with the canonical bundle
followed by the shift by the dimension of $X$;
$
 \bS(-) = (-) \otimes \omega_X [\dim X].
$

Let $d$ be a natural number.
An object $S$ of $\cC$ is said to be
\emph{spherical of dimension $d$}
if $\bS(S) \simeq S[d]$ and
\begin{align}
 \Hom^i(S,S) \cong
\begin{cases}
 \bfk & i = 0, d, \\
 0 & \text{otherwise}.
\end{cases}
\end{align}
The \emph{spherical twist}
along $S$
is the functorial cone
$T_S$
over the evaluation morphism
$
\ev \colon \hom(S, -) \otimes S \to \id_\cC.
$
It is
an autoequivalence
of the dg category
$\per \cC$
of perfect $\cC$-modules,
whose inverse
is given by the \emph{dual twist}
$T_S^\dual$
defined as the cone
over the coevaluation morphism
$
\coev \colon \id_\cC \to \hom(-,S)^\dual \otimes S
$
shifted by $-1$
\cite{MR1831820}.

Given 
a sequence
$
\bfS = (S_i)_{i=1}^\ell
$
of spherical objects,
let
$
\cB = \cB_\bfS
$
be the full subcategory of $\cC$
consisting of $\{ S_i \}_{i=1}^\ell$.
We identify
$\per \cB$
with
the full idempotent-complete pretriangulated subcategory of $\cC$
generated by $\cB$.
\begin{definition}
A sequence
$
\bfS = (S_i)_{i=1}^\ell
$
of spherical objects
is 
a \emph{spherical collection}
if
the Serre functor of $\cC$ preserves $\per \cB$
and acts as the shift functor there;
$
  \bS_\cC|_{\per \cB} \simeq [d].
$
\end{definition}
The number $\ell$ of objects in the collection
is called the \emph{length} of the collection.


The \emph{directed subcategory}
$\cA = \cA_{\bfS}$
of $\cB$
is the dg category whose set of objects is
$\{ E_i \}_{i=1}^\ell$
and
whose dg vector spaces of morphisms are given by
\begin{align}
 \hom_{\cA}(E_i, E_j) \coloneqq
\begin{cases}
 \hom_{\cB}(S_i, S_j) & i < j,\\
 \bfk \cdot \id_{E_i} & i=j, \\
 0 & i>j.
\end{cases}
\end{align}
The differentials and the compositions of morphisms in $\cA$
are inherited from that of $\cB$
in the obvious way.

Since the collection $\bfS$ is spherical,
the $\cA$-bimodule $\cB/\cA$,
defined as the cone of the natural morphism
$\cA \to \cB$ of $\cA$-bimodules,
is isomorphic to the shift
$
\cA^*[-d]
$
of the graph
$
\cA^* \coloneqq \hom_\bk(\cA, \bk)
$
of the Serre functor.
In other words,
the category $\cB$ is a
\emph{noncommutative anticanonical divisor}
of $\cA$
in the sense of
\cite{MR3727564}.

The sequence
$\lb E_i \rb_{i=1}^\ell$
is a full exceptional collection of $\per \cA$.
The existence of a full exceptional collection
implies that $\per \cA$ is proper and \emph{smooth}
(in the sense that the diagonal bimodule is perfect),
so that $\per \cA$ has a Serre functor.

Let
$
\iota_* \colon \per \cB \to \per \cA
$
be the dg functor
sending a $\cB$-module
to the same module
considered as an $\cA$-module
by the embedding $\cA \hookrightarrow \cB$.
The graph of the left adjoint functor
$
 \iota^* \colon \per \cA \to \per \cB
$
is $\cB$
considered as an $\cA^\op \otimes \cB$-module.
The right adjoint functor
is given by
$
  \iota^!
   = \bS_\cB \circ \iota^* \circ \bS_\cA^{-1}
  \simeq \iota^* \circ \bS_\cA^{-1} [ d ].
$

The \emph{dual twist} $T^\dual_{\iota_*}$
along $\iota_*$ is defined as
the shift by $-1$ of
the cone over the unit
$
\id_\cA \to \iota_* \circ \iota^*
$
of the adjunction.
The graph of the functor $\iota_* \circ \iota^*$
is $\cB$ considered as an $\cA$-bimodule,
so that the graph
of the dual twist
is the $\cA$-bimodule $\cB/\cA[-1]$,
which is the graph of $\bS[-d-1]$.
The \emph{twist} $T_{\iota_*}$ along $\iota_*$,
defined as the cone over the counit
$
\iota_* \circ \iota^! \to \id_\cA
$
of the adjunction,
is right adjoint to
$T^\dual_{\iota_*}$
by \cite[Proposition 5.3]{MR3692883},
and hence is
quasi-isomorphic to $\bS^{-1}[d+1]$.

The \emph{dual cotwist} $C^\dual_{\iota_*}$ along $\iota_*$,
defined as the cone over the counit
$
\iota^* \circ \iota_* \to \id_\cB,
$
is described explicitly as the (iterated) cone of
\begin{multline}
\hom(S_\ell, -)
\otimes \hom(S_{\ell-1}, S_\ell)
\otimes \cdots
\otimes \hom(S_1, S_2)
\otimes S_1
\to
\cdots\\
\to
\bigoplus_{1 \le i_1 < i_2 \le \ell}
\hom(S_{i_2}, -) \otimes \hom(S_{i_1}, S_{i_2}) \otimes S_{i_1}
\to
\bigoplus_{i=1}^\ell
\hom(S_i, -) \otimes S_i
\to
\id_\cB,
\end{multline}
which is quasi-isomorphic to the composition
of spherical twists;
\begin{align}
C^\dual_{\iota_*} \simeq
T_{S_1} \circ T_{S_2} \circ \cdots \circ T_{S_\ell}.
\end{align}
Since both $T_{\iota_*}$ and $C^\dual_{\iota_*}$ are equivalences,
the functor $\iota_*$ is spherical
in the sense of \cite[Theorem 1.1]{MR3692883}.
We expect that
$\iota_*$ has
a \emph{relative Calabi--Yau structure}
in the sense of \cite{MR3911626},
which is a slight strengthening
of an identification
of the spherical twist
with a shift of the inverse Serre functor.

For $1 \le i \le \ell-1$,
the
\emph{right mutation}
of a spherical collection
$
 \bfS = \lb S_j \rb_{j=1}^\ell
$
at position $i$
is defined by
\begin{align}
 R_i (\bfS) \coloneqq
  ( S_1, \ldots, S_{i-1}, S_{i+1}, T_{S_{i+1}}^\dual S_i, S_{i+2}, \ldots, S_\ell ).
\end{align}
The inverse operation is
the \emph{left mutation}
defined by
\begin{align}
 L_i (\bfS) \coloneqq
  \lb S_1, \ldots, S_{i-1}, T_{S_i} S_{i+1}, S_i, S_{i+2}, \ldots, S_\ell \rb.
\end{align}

The \emph{left mutation} $L_i(\cA)$ of $\cA$
at position $i$
is the full subcategory of $\per \cA$
consisting of $E_j$ for $j \ne i+1$
and
\begin{align}
 L_{E_i} E_{i+1} \coloneqq \Cone \lb \hom(E_i, E_{i+1}) \otimes E_i \xto{\ev} E_{i+1} \rb.
\end{align}
%
The passage to the directed subcategory commutes with mutations;
$
\cA_{L_i(\bfS)} \simeq L_i \lb \cA_{\bfS} \rb.
$

The \emph{helix}
generated by an exceptional collection $\lb E_i \rb_{i=1}^\ell$
is the sequence $\lb E_i \rb_{i \in \bZ}$
of exceptional objects of $\per \cA$
satisfying
\begin{align}
 E_{i-\ell} = L_{E_{i-\ell+1}} \circ L_{E_{i-\ell+2}}
  \circ \cdots \circ L_{E_{i-1}}(E_i)[ - d - 1 ]
\end{align}
for any $i \in \bZ$.
The shift is chosen in such a way that
$E_{i-\ell} \cong \bS(E_i)[-d-1]$,
so that
if
$\per \cA \simeq D^b \coh X$
for a smooth projective variety $X$
of dimension $d+1$,
then one has
$E_{i-\ell} \cong E_i \otimes \omega_X$.
The length $\ell$ of the exceptional collection
generating the helix is called the \emph{period} of the helix.
For any $i \in \bZ$,
the exceptional collection $\lb E_j \rb_{j=i}^{i+\ell-1}$
consisting of $\ell$ consecutive objects in the helix
is called a \emph{foundation}
of the helix.
We say that
a helix $\lb E_i \rb_{i \in \bZ}$ is \emph{acyclic}
if
$
 \Hom^{\ne 0}(E_i, E_j) = 0
$
for any $-\infty < i < j < \infty$.

Given an invertible $\cA$-bimodule $\theta$,
the \emph{tensor algebra} over $\cA$ is defined by
\begin{align}
T_\cA(\theta)
&\coloneqq \bigoplus_{i=0}^\infty \theta^{\otimes_\cA i} \\
&= \cA \oplus \theta
\oplus \left( \theta \otimes_\cA \theta  \right)
\oplus \cdots \\
&\simeq \bigoplus_{i=0}^\infty \bigoplus_{j,k=1}^\ell
\hom\left( E_j, \theta^i \left( E_k \right) \right).
\end{align}

In addition to the cohomological grading
coming from that of $\cA$ and $\theta$,
the tensor algebra $T_\cA(\theta)$
has an additional grading
such that $\theta^{\otimes_\cA i}$ has degree $i$.
The positively graded part
$
T_\cA(\theta)^+
\coloneqq \bigoplus_{i=1}^\infty \theta^{\otimes_\cA i}
\simeq \theta \otimes_\cA T_\cA(\theta)
$
is
the kernel
(the cone shifted by $-1$)
of the natural morphism
$
\sigma^\sharp \colon T_\cA(\theta) \to \cA.
$

The category
$\per T_\cA(\theta)$
is a generalization
of the derived category of coherent sheaves
on the total space of a line bundle,
and
the pull-back
$
\sigma_* \colon \per \cA \to \per T_\cA(\theta)
$
along $\sigma^\sharp$
is a generalization
of the push-forward
along the zero section.
Its left adjoint
$
\sigma^* \colon \per T_\cA(\theta) \to \per \cA
$
is given by tensoring $\cA$ over $T_\cA(\theta)$.
Since the graph of
$
  \sigma^* \circ \sigma_*
$
is
\begin{align}
  \Cone \left( \theta \otimes_\cA T_\cA(\theta) \to T_\cA(\theta) \right)
   \otimes_{T_\cA(\theta)} \cA
  \simeq \Cone \left( \theta \to \cA \right),
\end{align}
the dual cotwist along $\sigma_*$
is the autoequivalence $\theta[1]$
of $\per \cA$.
See also \cite{MR3805198}
for a related construction.

The push-forward
$
\pi^* \coloneqq (-) \otimes_\cA T_\cA(\theta)
\colon \per \cA \to \per T_\cA(\theta)
$
along the structure morphism
$
 \pi^\sharp \colon \cA \to T_\cA(\theta)
$
of the tensor algebra
is a generalization of the pull-back
along the projection to the base space.
For any pair $(E,F)$ of objects of $\per \cA$,
one has
\begin{align}
\hom \lb \sigma_* E, \pi^* F \rb
&\simeq
\hom
\lb
 \pi^* E
 \otimes_{T_\cA(\theta)}
 \Cone
 \lb
  \theta \otimes_\cA T_\cA(\theta) \to T_\cA(\theta)
 \rb
 ,
 \pi^* F
\rb
\\
&\simeq
\hom
\lb
 \pi^* E
 ,
 \pi^* F
 \otimes_{T_\cA(\theta)}
 \Cone
 \lb
  T_\cA(\theta) \to \theta^{-1} \otimes_\cA T_\cA(\theta)
 \rb[-1]
\rb
\\
&\simeq
\hom
\lb
 \pi^* E
 ,
 \sigma_* (\theta^{-1}(F))
\rb[-1]
\\
&\simeq
\hom
\lb
 E
 ,
 \theta^{-1}(F)
\rb[-1].
\end{align}

The dg category
$\gr T_\cA(\theta)$
of graded perfect $T_\cA(\theta)$-modules
is a generalization
of the derived category of equivariant coherent sheaves
on the total space of a line bundle
with respect to the $\Gm$-action
by fiberwise dilation.
The dg $\bZ$-algebra
$
 A = \bigoplus_{i,j=-\infty}^\infty A_{ij}
$
defined by
\begin{align}
 A_{i+k\ell,j+m\ell} = \hom(\theta^m(E_j), \theta^k(E_i))
\end{align}
for $1 \le i,j \le \ell$ and $-\infty < m \le k < \infty$
is $\ell$-periodic,
and the $1$-periodic dg $\bZ$-algebra
$
 A' = \bigoplus_{i,j=-\infty}^\infty A'_{ij}
$
defined by
\begin{align}
 A'_{km} = \bigoplus_{i=(k-1)\ell+1}^{k\ell} \bigoplus_{j=(m-1)\ell+1}^{m\ell} A_{ij}
\end{align}
is the dg $\bZ$-algebra
associated to the $\bZ$-graded algebra $T_\cA(\theta)$.
One has quasi-equivalences
\begin{align} \label{eq:gr=mod}
 \gr T_\cA(\theta)
 \simeq  
 \module A'
 \simeq
 \module A
\end{align}
of dg categories.

Let $(F_i)_{i=\ell}^1$ be the exceptional collection in $\per \cA$
right dual to $(E_i)_{i=1}^\ell$
so that $\hom(E_i,F_j) \simeq \bk$ if $i=j$ and $0$ otherwise
\cite[Section 7]{MR992977}.
Then
$E_i$ (resp. $F_i$) represents
the $i$-th projective (resp. simple) $\cA$-module,
and
$\pi^* E_i$ (resp. $\sigma_* F_i$) represents
the $i$-th projective (resp. simple) $T_\cA(\theta)$-module.

The dg algebra $T_\cA(\theta)$
in the case
$
\theta = \bS^{-1}[d+1]
$
is known as
the \emph{derived $(d+1)$-preprojective algebra},
which is a model of the \emph{trivial Calabi--Yau completion}
\cite{MR2795754};
\begin{align}
  \Pi_{d+1}(\cA)
  &\coloneqq T_\cA(\bS^{-1}[d+1]).
\end{align}
It is given explicitly as
\begin{align}
 \bigoplus_{i=0}^\infty \bigoplus_{j,k=1}^\ell
  \hom\left( E_j, \bS^{-i} \left( E_k \right) [i(d+1)] \right)
  \simeq \bigoplus_{j=1}^\ell \bigoplus_{k=1}^\infty
  \hom\left( E_j, E_k \right),
\end{align}
and also known as the \emph{rolled-up helix dg algebra}.
The dg $\bZ$-algebra $A$ in this case will be called
the \emph{helix dg algebra}.

For $1 \le i, j \le \ell$,
one has
\begin{align}
 \hom \lb \sigma_* F_i, \pi^* E_j \rb
 &\simeq
 \hom
 \lb
 F_i
 ,
 \bS(E_j)[-d-1]
 \rb[-1] \\
 &\simeq
 \hom(E_j,F_i)[-d-2] \\
 &\simeq
 \begin{cases}
  \bk[-d-2] & i=j, \\
  0 & \text{otherwise}.
 \end{cases}
 \label{eq:dg_AS-Gorenstein}
\end{align}

An object of $\gr T_\cA(\theta)$ is said to be
a \emph{torsion}
if the grading is bounded.
Let
$\tor T_\cA(\theta)$
be the full subcategory
of $\gr T_\cA(\theta)$
consisting of torsion modules,
and
$
\qgr T_\cA(\theta)
\coloneqq \gr T_\cA(\theta) / \tor T_\cA(\theta)
$
be the dg quotient.

\begin{proposition}
One has an equivalence
$
\qgr T_\cA(\theta) \simeq \per \cA.
$
\end{proposition}

\begin{proof}
One has a semiorthogonal decomposition
\begin{align}
  \tor T_\cA(\theta)
  = \la \cS_i \ra_{i \in \bZ},
\end{align}
where $\cS_i \simeq \per \cA$ is the admissible subcategory
consisting of graded modules
concentrated in degree $i$.
One has
\begin{align}
  \gr T_\cA(\theta)
   = \la \ldots, \cS_{-2}, \cS_{-1},
    \lb \gr T_\cA(\theta) \rb_{\ge 0} \ra
\end{align}
just as in \cite[Lemma 14]{MR2641200},
where
$
\lb \gr T_\cA(\theta) \rb_{\ge 0}
$
is the full subcategory of
$
\gr T_\cA(\theta)
$
consisting of non-negatively graded modules.
It follows from the quasi-isomorphism
\begin{align}
  \cA \simeq \Cone(T_\cA(\theta)^+ \to T_\cA(\theta))
\end{align}
of graded $T_\cA(\theta)$-modules
that
the free module
$T_\cA(\theta)$
is right orthogonal to
$
\la \cS_i \ra_{i \ge 0},
$
and that its shift
$
T_\cA(\theta)(i)
$
for any $i < 0$
is in the full pretriangulated subcategory
generated by $T_\cA(\theta)$ and
$
\la \cS_i \ra_{i \ge 0}.
$
Hence 
$T_\cA(\theta)$
generates the right orthogonal to
$
\la \cS_i \ra_{i \ge 0}
$
in
$
\lb \gr T_\cA(\theta) \rb_{\ge 0},
$
so that 
$
\qgr T_\cA(\theta)
$
is equivalent to the full subcategory
$
\la T_\cA(\theta) \ra
$
of $\gr T_\cA(\theta)$,
which in turn is equivalent to $\per \cA$.
\end{proof}

\begin{definition}
A \emph{spherical helix}
generated by a spherical collection $\lb S_i \rb_{i=1}^\ell$
is a sequence $\lb S_i \rb_{i \in \bZ}$
of spherical objects of $\cC$
satisfying
\begin{align}
 S_{i-\ell}
  = T_{S_{i-\ell+1}} \circ T_{S_{i-\ell+2}}
   \circ \cdots \circ T_{S_{i-1}}(S_i)[ - d - 1 ]
\end{align}
for any $i \in \bZ$.
\end{definition}

The length $\ell$ of the spherical collection
generating the spherical helix is called
the \emph{period} of the spherical helix.
For any $i \in \bZ$,
the spherical collection $\lb S_j \rb_{j=i}^{i+\ell-1}$
consisting of $\ell$ consecutive objects in the spherical helix
is called a \emph{foundation}
of the spherical helix $\cH$.

\begin{definition}
A spherical helix
$\lb S_i \rb_{i \in \bZ}$
is \emph{acyclic}
if
$
 \Hom^{\ne 0}(S_i, S_j) = 0
$
for any $-\infty < i < j < \infty$.
\end{definition}
One has
\begin{align}
\hom(S_i, S_j)
&\simeq \hom(\iota^* E_i, \iota^* E_j) \\
&\simeq \hom(E_i, \iota_* \iota^* E_j) \\
&\simeq \hom \left( E_i,
 \Cone \left( E_{j-\ell} \to E_j \right) \right),
\end{align}
which immediately yields the following:

\begin{theorem}
If the spherical helix
$\left( S_i \right)_{i \in \bZ}$
is acyclic,
then
the helix
$\left( E_i \right)_{i \in \bZ}$
is acyclic.
\end{theorem}

Recall that a connected $\bZ$-algebra
$
A=\bigoplus_{i,j=-\infty}^\infty A_{ij}
$
over $\bfk$
is said to be
\emph{AS-Gorenstein}
if
\begin{align}
  \sum_{j,k=-\infty}^\infty
  \dim \Hom^k(S_{i,A}, P_{j,A})
  =1
\end{align}
for any $i \in \bZ$,
where $S_{i,A}$ and $P_{i,A}$ are
the $i$-th simple and projective $A$-modules respectively.
It follows from \pref{eq:dg_AS-Gorenstein}
that the acyclicity of the helix $(E_i)_{i \in \bZ}$
implies the AS-Gorenstein property
of the helix $\bZ$-algebra $A$.
Since $\cA$ is smooth and $\theta$ is perfect,
the tensor algebra $T_\cA(\theta)$ is smooth,
so that $A$ has finite global dimension
(see e.g.~\cite[Theorem 4.8]{MR2795754}).
The Hilbert polynoimal of $A$ is determined
by the projective resolutions of the simples.


\emph{Acknowledgement}:
S.~O.~was partially supported by
JSPS Grant-in-Aid for Scientific Research
No.16H05994, 16H02141, 16H06337, 18H01120, 20H01797. 
K.~U.~was partially supported by
JSPS Grant-in-Aid for Scientific Research
No.16H03930.

\bibliographystyle{amsalpha}
\bibliography{bibs}

\def\cprime{$'$} \def\cprime{$'$}
\providecommand{\bysame}{\leavevmode\hbox to3em{\hrulefill}\thinspace}
\providecommand{\MR}{\relax\ifhmode\unskip\space\fi MR }
\providecommand{\MRhref}[2]{%
  \href{http://www.ams.org/mathscinet-getitem?mr=#1}{#2}
}
\providecommand{\href}[2]{#2}
\begin{thebibliography}{VdB11}

\bibitem[AL17]{MR3692883}
Rina Anno and Timothy Logvinenko, \emph{Spherical {DG}-functors}, J. Eur. Math.
  Soc. (JEMS) \textbf{19} (2017), no.~9, 2577--2656. \MR{3692883}

\bibitem[BD19]{MR3911626}
Christopher Brav and Tobias Dyckerhoff, \emph{Relative {C}alabi-{Y}au
  structures}, Compos. Math. \textbf{155} (2019), no.~2, 372--412. \MR{3911626}

\bibitem[BK89]{MR1039961}
A.~I. Bondal and M.~M. Kapranov, \emph{Representable functors, {S}erre
  functors, and reconstructions}, Izv. Akad. Nauk SSSR Ser. Mat. \textbf{53}
  (1989), no.~6, 1183--1205, 1337. \MR{MR1039961 (91b:14013)}

\bibitem[Bon89]{MR992977}
A.~I. Bondal, \emph{Representations of associative algebras and coherent
  sheaves}, Izv. Akad. Nauk SSSR Ser. Mat. \textbf{53} (1989), no.~1, 25--44.
  \MR{MR992977 (90i:14017)}

\bibitem[BP93]{MR1230966}
A.~I. Bondal and A.~E. Polishchuk, \emph{Homological properties of associative
  algebras: the method of helices}, Izv. Ross. Akad. Nauk Ser. Mat. \textbf{57}
  (1993), no.~2, 3--50. \MR{1230966 (94m:16011)}

\bibitem[Kel06]{MR2275593}
Bernhard Keller, \emph{On differential graded categories}, International
  {C}ongress of {M}athematicians. {V}ol. {II}, Eur. Math. Soc., Z\"urich, 2006,
  pp.~151--190. \MR{MR2275593 (2008g:18015)}

\bibitem[Kel11]{MR2795754}
\bysame, \emph{Deformed {C}alabi-{Y}au completions}, J. Reine Angew. Math.
  \textbf{654} (2011), 125--180, With an appendix by Michel Van den Bergh.
  \MR{2795754}

\bibitem[Orl09]{MR2641200}
Dmitri Orlov, \emph{Derived categories of coherent sheaves and triangulated
  categories of singularities}, Algebra, arithmetic, and geometry: in honor of
  {Y}u. {I}. {M}anin. {V}ol. {II}, Progr. Math., vol. 270, Birkh\"auser Boston
  Inc., Boston, MA, 2009, pp.~503--531. \MR{2641200 (2011c:14050)}

\bibitem[Seg18]{MR3805198}
Ed~Segal, \emph{All autoequivalences are spherical twists}, Int. Math. Res.
  Not. IMRN (2018), no.~10, 3137--3154. \MR{3805198}

\bibitem[Sei17]{MR3727564}
Paul Seidel, \emph{Fukaya {$A_\infty$}-structures associated to {L}efschetz
  fibrations. {II}}, Algebra, geometry, and physics in the 21st century, Progr.
  Math., vol. 324, Birkh\"{a}user/Springer, Cham, 2017, pp.~295--364.
  \MR{3727564}

\bibitem[ST01]{MR1831820}
Paul Seidel and Richard Thomas, \emph{Braid group actions on derived categories
  of coherent sheaves}, Duke Math. J. \textbf{108} (2001), no.~1, 37--108.
  \MR{MR1831820 (2002e:14030)}

\bibitem[VdB11]{MR2836401}
Michel Van~den Bergh, \emph{Noncommutative quadrics}, International Mathematics
  Research Notices. IMRN (2011), no.~17, 3983--4026. \MR{MR2836401
  (2012m:14004)}

\end{thebibliography}

\end{document}